\documentclass[runningheads,orivec]{llncs}
\usepackage[T1]{fontenc}
\usepackage{graphicx}
\usepackage{amsmath,tikz,float,amssymb}
\usepackage[shortlabels]{enumitem}
\setenumerate[1]{itemsep=1pt,parsep=\parskip,topsep=3pt}
\usepackage[hidelinks]{hyperref}
\usepackage{cleveref}
\makeatletter
\renewcommand{\@Opargbegintheorem}[4]{%
  #4\trivlist\item[\hskip\labelsep{#3#2\@thmcounterend}]}
\makeatother

\begin{document}

\title{On the minimum doubly resolving set problem in line graphs}

\author{Qingjie Ye}

\authorrunning{Q. Ye}

\institute{School of Mathematical Sciences, Key Laboratory of MEA (Ministry of Education),
Shanghai Key Laboratory of PMMP, East China Normal University, Shanghai, China\\
\email{qjye@math.ecnu.edu.cn}
}

\maketitle

\begin{abstract}
   Given a connected graph $G$ with at least three vertices, let $d_G(u,v)$ denote the distance between vertices $u,v\in V(G)$. A  subset $S\subseteq V$ is called a doubly resolving set (DRS) of $G$ if for any two distinct vertices $u, v \in V(G)$, there exists a pair $\{x,y\}\subseteq S$ such that $d_G(u,x)-d_G(u,y)\neq d_G(v,x)-d_G(v,y)$. This paper studies the minimum cardinality of a DRS in the line graph of $G$, denoted by $\Psi(L(G))$. First, we prove that computing $\Psi(L(G))$ is NP-hard, even when $G$ is a bipartite graph. Second, we establish that  $\lceil \log_2 (1+\Delta(G))\rceil \le \Psi(L(G)) \le |V(G)| - 1$ holds for all $G$ with maximum degree  $\Delta(G)$, and show that both inequalities are tight. Finally, we determine the exact value of $\Psi(L(G))$ provided $G$ is a tree.

\end{abstract}

\section{Introduction}
Let $G$ be a finite, connected, simple and undirected graph with vertex set $V=V(G)$ and edge set $E=E(G)$, where $|V|\ge2$. The distance, i.e., the number of edges in a shortest path, between vertices $u$ and $v$ is denoted by $d_G(u,v)$. 

The metric dimension problem was independently introduced by Slater~\cite{slater1975leaves}, Harary and Melter~\cite{harary1976metric}. A vertex subset $S(\subseteq V)$ resolves  $G$ if every vertex is uniquely determined by its vector of distances to the vertices in $S$. More formally, a vertex $x\in V$ \emph{resolves} two vertices $u,v\in V$ if $d_G(u,x)\ne d_G(v,x)$. A vertex subset $S$ is a \emph{resolving set} of $G$ if every two vertices in $G$ are resolved by some vertex of $S$. A resolving set $S$ of $G$ with the minimum cardinality is a \emph{metric basis} of $G$, and the minimum cardinality is referred to as the \emph{metric dimension} of $G$, denoted by $\mu(G)$. The \emph{metric dimension problem} is to determine $\mu(G)$ for an input graph $G$. The study of the metric dimension is motivated by its applications in diverse areas, including network discovery and verification~\cite{beerliova2005network}, robot navigation~\cite{khuller1996landmarks}, chemistry~\cite{chartrand2000resolvability}, and machine learning on biological data~\cite{tillquist2019low}.

Doubly resolving sets were introduced by C\'{a}ceres et~al.~\cite{caceres2007metric} as a tool for investigating the resolving sets of Cartesian products of graphs. 
Given $x,y,u,v\in V$, we say that $\{x,y\}$ \emph{doubly resolves} $\{u,v\}$, if $d_G(u,x)-d_G(u,y)\neq d_G(v,x)-d_G(v,y)$. A vertex subset $S$ of $G$ is a \emph{doubly resolving set} (DRS) of $G$ if every pair of distinct vertices in $G$ is doubly resolved by some pair of vertices in $S$.  Let $\Psi(G)$ denote the minimum cardinality of a DRS of $G$. The \emph{minimum doubly resolving set problem} (or minimum DRS problem for short) is to determine $\Psi(G)$ for an input graph $G$. The problem of finding a  minimum DRS is equivalent to locating the source of a diffusion in complex networks~\cite{chen2016approximation}.

Both the metric dimension problem and the minimum DRS problem are NP-hard in general. The hardness proof for the former problem was given in \cite{khuller1996landmarks}, and that for the latter problem was given in \cite{kratica2009computing}. Furthermore, Epstein et~al.~\cite{epstein2015weighted} proved that the metric dimension problem is NP-hard even for split graphs, bipartite graphs, co-bipartite graphs, and line graphs of bipartite graphs. 
Lu and Ye~\cite{lu2022bridge} proved that the minimum DRS problem is NP-hard even for split graphs, bipartite graphs, and co-bipartite graphs. The two problems have many interesting theoretical properties, which are out of the scope of this paper. The interested reader is referred, e.g. to \cite{bailey2011base,chen2016approximation,jernando2010extremal,jiang2019metric,kratica2012minimalB,laird2020resolvability,nasir2019edge,tillquist2023getting}.

Let $T$ be a tree, where  degree-1 vertices are called \emph{leaves} and vertices of degree 3 or above are called \emph{major vertices}. Let $u$ be a leaf and $v$ be a major vertex. We say that $u$ is a \emph{terminal vertex} of  $v$   if $d_T (u, v) < d_T (u,w)$ for every other major vertex $w$ of $T$. A major vertex is an \emph{exterior major vertex} (a \emph{strong exterior major vertex}) if it has at least a terminal vertex (a leaf neighbor) in $T$. Let $\sigma (T)$, $ex(T)$ and $ex'(T)$  denote the number of leaves, the number of exterior major vertices, and the number of strong exterior major vertices in $T$, respectively.
The metric dimension problem of a tree has been solved by Slater~\cite{slater1975leaves}, Harary and Melter~\cite{harary1976metric}, and Chartrand et~al.~\cite{chartrand2000resolvability}.
\begin{theorem}\label{rs}
    If $T$ is a tree that is not a path, then $\mu(T)=\sigma(T) - ex(T)$. If $T$ is a path, then $\mu(T) = 1$.
\end{theorem}

For the minimum DRS problem on tree $T$, C\'{a}ceres et~al.~\cite{caceres2007metric} established $\Psi(T)=\sigma(T)$.

Graph $G$ is associated with its line graph, denoted $L(G)$, which is a graph where each vertex corresponds to an edge of $G$. Two vertices in $L(G)$ are adjacent if and only if their corresponding edges in $G$ share a common end vertex. We recall the NP-hardness result for the metric dimension problem in line graphs, as established by Epstein et~al.~\cite{epstein2015weighted}.
\begin{theorem} \label{th:mdhard}
    Given a bipartite graph $G$, determining $\mu(L(G))$ is NP-hard. 
\end{theorem}

Feng et~al.~\cite{feng2013metric} studied the metric dimension problem of line graphs and obtained the following two theorems.
\begin{theorem}\label{line}
    If $G$ is a connected graph with at least five vertices, then
    \[\lceil \log_2 \Delta(G)\rceil \le \mu(L(G)) \le |V(G)| - 2,\]
    where $\Delta(G)$ is the maximum degree of $G$.
\end{theorem}
\begin{theorem}\label{tree-rs}
    If $T$ is a tree, then $\mu(L(T))=\mu(T)$.
\end{theorem}

\paragraph{Our contributions.}
A natural question arises: Does the minimum DRS problem in line graphs admit conclusions similar to those in Theorems \ref{th:mdhard}, \ref{line}, and \ref{tree-rs}? This paper focuses on addressing this question affirmatively. 
\begin{itemize}
    \item In Section \ref{sec:np}, we prove that the minimum DRS problem is NP-hard for the line graphs of bipartite graphs (Theorem \ref{th:hard}) -- a counterpart to Theorem~\ref{th:mdhard}. Our proof is inspired by Epstein et~al.~\cite{epstein2015weighted}.
\item    In Section \ref{sec:linegraph}, complementary to Theorem \ref{line}, we provide similar yet distinct lower bound $\lceil \log_2 (1+\Delta(G))\rceil$ and upper bound $|V(G)|-1$ for $\Psi(L(G))$ (see Theorem~\ref{main}). Both bounds are shown to be tight. Moreover,  the inequalities in Theorem \ref{line} are valid only for the case $|V(G)|\ge5$; in contrast, our results also apply to the cases where $|V(G)| \in \{3,4\}$. Note that for a connected graph $G$, $|V(L(G))| \geq 2$ if and only if $|V(G)| \geq 3$, so our theorem covers all possible cases. 
\item In Section \ref{sec:tree}, we show that  $\Psi(L(T))=\sigma(T)-ex'(T)$ holds for all trees $T$ (see Theorem \ref{treethm}), which implies an $O(|V(T)|)$ time algorithm for finding a minimum DRS in $L(T)$. On the one hand, the equation shows the similarity to $\mu(L(T))= \sigma(T)-ex(T)$ when $T$ is not a path (as indicated by Theorems  \ref{rs} and \ref{tree-rs}), on the other hand, while $T$ and $L(T)$ share the same value of metric dimension, the sizes of their minimum DRSs differ by $ex'(T)$ because $\Psi(T)=\sigma(T)$ \cite{caceres2007metric}. For example $\Psi(K_{1,n})=n\neq n-1=\Psi(K_n)=\Psi(L(K_{1,n}))$. 
\end{itemize}

\section{NP-hardness}\label{sec:np}
In this section, we prove that the minimum DRS problem is NP-hard even for line graphs of bipartite graphs. The proof is an extension of the work in \cite{epstein2015weighted}, which showed the NP-hardness for the metric dimension problem. Following \cite{epstein2015weighted}, we also employ a reduction from the 3-dimensional matching problem, which is well-known to be NP-hard. 

The $3$-dimensional matching problem is defined as follows. Given three disjoint sets $A,B,C$ such that $|A|=|B|=|C|=n$, and a set of triples $S \subseteq A \times B \times C$, is there a subset $S'\subseteq S$ such that each element of $A \cup B \cup C$ occurs in exactly one of the triples in $S'$. Based on the instance $(A,B,C,S)$ of the $3$-dimensional matching problem, Epstein et~al.~\cite{epstein2015weighted} constructed a bipartite graph $G$ and established a connection between the solution of the 3-dimensional matching instance and the sizes of resolving sets in $L(G)$. In this paper, we use the same bipartite graph $G$ and relate the solution to the sizes of DRSs in $L(G)$. Our main contribution is to demonstrate that the resolving set they identified for the metric dimension problem on $L(G)$ is, in fact, a DRS of $L(G)$.

Let $N:=2^{12}n$ and $n':=nN$. Let $\mathcal{A}:=\bigcup_{i=1}^N A_i$, $\mathcal{B}:=\bigcup_{i=1}^N B_i$ and $\mathcal{C}:=\bigcup_{i=1}^N C_i$, where $A_i,B_i,C_i$ are the copies of $A,B,C$ respectively. Let $\mathcal{S}:=\bigcup_{i=1}^N S_i$, where $S_i\,(\subseteq A_i\times B_i\times C_i)$ is a copy of $S$. 
Let $\tau:=|\mathcal{S}|$ and $\lambda:=\lceil \log_2 \tau \rceil$. In the bipartite graph $G=(I\cup J,E)$ constructed by Epstein et~al.~\cite{epstein2015weighted}, the vertex set is partitioned into two independent sets:
\begin{eqnarray*}
    I&=&\mathcal{S} \cup \{s_{\mathcal{A}}, s_{\mathcal{B}}, s_{\mathcal{C}}, s_{\mathcal{D}}\} \cup \{d'_0,d'_1,\dots,d'_{\lambda-1}\},\\
    J&=&\mathcal{A}\cup \mathcal{B}\cup \mathcal{C}\cup \{s'_0,s'_1,\dots,s'_{\tau-1}\}\cup \{s'_{\mathcal{A}}, s'_{\mathcal{B}}, s'_{\mathcal{C}}, s'_{\mathcal{D}}\} \cup \{d_0,d_1,\dots,d_{\lambda-1}\}.
\end{eqnarray*}

For any $u \in J$ and $v \in I$, $uv$ is an edge in $ E$ if and only if $u$ and $v$ satisfy one of the thirteen following cases (see Fig. \ref{etunpc} in Appendix \ref{apx:graph} for an illustration): (1) $u \in \mathcal{A}$ and $v = s_{\mathcal{A}}$; (2)  $u \in \mathcal{B}$ and $v = s_{\mathcal{B}}$; (3) $u \in \mathcal{C}$ and $v = s_{\mathcal{C}}$; (4) $u \in \mathcal{A}\cup \mathcal{B}\cup \mathcal{C}$ and $v = s_{\mathcal{D}}$;
  (5) $u \in \{a,b,c\}$ and $v = (a,b,c) \in \mathcal{S}$;
    (6) $u = d_i$ and $v = s_j$ such that $\lfloor j/2^i \rfloor \bmod 2 = 1$;
    (7) $u = d_i$ and $v = s_{\mathcal{D}}$;
    (8) $u = s'_i$ and $v = s_i$;
   (9) $u = s'_{\mathcal{A}}$ and $v = s_{\mathcal{A}}$;
    (10) $u = s'_{\mathcal{B}}$ and $v = s_{\mathcal{B}}$;
    (11) $u = s'_{\mathcal{C}}$ and $v = s_{\mathcal{C}}$;
    (12) $u = s'_{\mathcal{D}}$ and $v = s_{\mathcal{D}}$; (13) $u = d_i$ and $v = d'_i$.

It is clear that the construction can be done in polynomial time. Let $K=n+\lambda+4$. Epstein et~al.~\cite{epstein2015weighted} has proved that there exists a three-dimensional matching $S'\subseteq S$ if and only if $L(G)$ has a resolving set $R$ such that $|R| \leq K$. A similar statement also holds for the minimum DRS problem.

\begin{lemma}\label{edrsnpc}
    \begin{enumerate}[(a)]
    \item If $L(G)$ has a DRS $R$ with $|R|\le K$, then there is a $3$-dimensional matching $S'\subseteq S$. \label{enpdrsa}
    \item If there is a $3$-dimensional matching $S'\subseteq S$, then the corresponding $\mathcal{S}'\subseteq \mathcal{S}$  gives rise to the set
    \[R=\{s_is_i':s_i\in \mathcal{S}'\}\cup \{s_{\mathcal{A}}s'_{\mathcal{A}}, s_{\mathcal{B}}s'_{\mathcal{B}}, s_{\mathcal{C}}s'_{\mathcal{C}}, s_{\mathcal{D}}s'_{\mathcal{D}}\}\cup \{d_0d'_0,d_1d'_1,\dots,d_{\lambda-1}d'_{\lambda-1}\},\]
    which is a DRS of $L(G)$ with $|R|=K$. \label{enpdrsb}
    \end{enumerate}
\end{lemma}

We provide the proof in Appendix \ref{apx:lma:edrsnpc}. Consequently, we establish the NP-hardness of computing $\Psi(L(G))$, even when $G$ is restricted to be a bipartite graph.
\begin{theorem} \label{th:hard}
    Given a bipartite graph $G$, determining $\Psi(L(G))$ is NP-hard. 
\end{theorem}

\section{General line graphs}\label{sec:linegraph}
Inspired by the inequalities $\lceil \log_2 \Delta(G)\rceil \le \mu(L(G)) \le |V(G)| - 2$ for the metric dimension of line graphs \cite{feng2013metric}, in this section, we establish similar and tight bounds for   $\Psi(L(G))$ in the minimum DRS problem. 
\begin{theorem}\label{main}
    If $G$ is a connected graph with at least three vertices, then
    \[\lceil \log_2 (1+\Delta(G))\rceil \le \Psi(L(G)) \le |V(G)| - 1.\]
\end{theorem}
We prove the lower bound and $\lceil \log_2 (1+\Delta(G))\rceil$ the upper bound $|V(G)| - 1$ stated in Theorem~\ref{main}, and discuss their tightness in the following two subsections, respectively.
\subsection{Lower bound}
To verify the lower bound, we need the following basic property on DRSs derived by Kratica et~al.~\cite{kratica2009computing}.
\begin{lemma}
\label{small}
    Let $G'$ be a connected graph with $|V(G')|\ge2$ and $S=\{u_1,u_2,\dots,u_m\}$ be a DRS of $G'$. Then for each pair of distinct vertices $x,y\in V(G')$, there exists $u_j\in S$, such that $d_{G'}(x,u_j)-d_{G'}(x,u_1)\neq d_{G'}(y,u_j)-d_{G'}(y,u_1)$.
\end{lemma}

The above lemma  means that the vector $F(x|S):=(d_{G'}(x,u_2)-d_{G'}(x,u_1),\dots,$ $d_{G'}(x,u_m)-d_{G'}(x,u_1))$ uniquely determines $x$  if $S=\{u_1,u_2,\dots,u_m\}$ is a DRS of $G'$.

\begin{proof}[Proof of Theorem \ref{main} (Lower bound)]  Let $u$ be a vertex with degree $\Delta(G)$, and let $F_u = \{f_1, f_2, \dots, f_{\Delta(G)}\}$ be the set of edges incident to $u$. Let $S = \{e_1, \dots, e_s\}$ be a minimum DRS of $L(G)$, where $s=\Psi(L(G))$. For any edge $f_i \in F_u$, let $a_{ik} = d_{L(G)}(f_i, e_k)$. Since any two edges $f_i, f_j \in F_u$ are adjacent, for any fixed $k$, the distance $a_{ik}$ can only take two consecutive integer values.  Let $d_k = \min_{i} \{a_{ik}\}$. Then for any $i$ and $k$, $a_{ik} \in \{d_k, d_k + 1\}$.

    For each $f_i$, define its distance difference vector as $F(f_i|S) = (a_{i2} - a_{i1}, a_{i3} - a_{i1}, \dots, a_{is} - a_{i1})$. Let $\mathcal{F} = \{F(f_i|S) : i \in \{1, \dots, \Delta(G)\}\}$ be the set of all such vectors. By Lemma \ref{small}, all vectors in $\mathcal{F}$ must be distinct, so $|\mathcal{F}| = \Delta(G)$. We partition $\mathcal{F}$ into two disjoint sets based on the value of $a_{i1}$: $\mathcal{F}_0=\{F(f_i|S) : a_{i1}=d_1\}$ and $\mathcal{F}_1=\{F(f_i|S) : a_{i1}=d_1+1\}$. 
    For any vector in $\mathcal{F}_0$, its $(k-1)$-th component is $a_{ik} - d_1 = (d_k - d_1) + (a_{ik}-d_k)$, where $a_{ik}-d_k \in \{0,1\}$.
    For any vector in $\mathcal{F}_1$, its $(k-1)$-th component is $a_{ik} - (d_1+1) = (d_k - d_1) + (a_{ik}-d_k-1)$, where $a_{ik}-d_k-1 \in \{-1,0\}$. This shows that $\mathcal{F}_0$ is a subset of $\mathcal{F}'_0 = \{(d_2-d_1+x_2, \dots, d_s-d_1+x_s) : x_k \in \{0,1\}\}$, and $\mathcal{F}_1$ is a subset of $\mathcal{F}'_1 = \{(d_2-d_1+x_2, \dots, d_s-d_1+x_s) : x_k \in \{-1,0\}\}$.
    The sizes of these sets are $|\mathcal{F}'_0| = |\mathcal{F}'_1| = 2^{s-1}$. Their intersection $\mathcal{F}'_0 \cap \mathcal{F}'_1$ occurs only when $x_k=0$ for all $k$, which corresponds to the single vector $(d_2-d_1, \dots, d_s-d_1)$. Using the principle of inclusion-exclusion, we can bound the size of $\mathcal{F}$:
    \[|\mathcal{F}| = |\mathcal{F}_0 \cup \mathcal{F}_1| \le |\mathcal{F}'_0 \cup \mathcal{F}'_1| = |\mathcal{F}'_0| + |\mathcal{F}'_1| - |\mathcal{F}'_0 \cap \mathcal{F}'_1| = 2^{s-1} + 2^{s-1} - 1 = 2^s - 1.\]
    Since $|\mathcal{F}| = \Delta(G)$ and $s = \Psi(L(G))$, we have $\Delta(G) \leq 2^{\Psi(L(G))} - 1$. Rearranging this gives the desired bound: $\Psi(L(G)) \ge \lceil \log_2 (1 + \Delta(G)) \rceil$.\qed
\end{proof}

\paragraph{Tightness.} The lower bound $\lceil \log_2 (1 + \Delta(G)) \rceil$ is tight. First, it is not difficult to see that equality holds when $G$ is a path, an odd cycle, or a $K_{1,3}$. More generally, we construct a graph $A_k$ (see Fig. \ref{ak} for an illustration) as follows:
\[V(A_k)=\{u\}\cup \{v_1,\dots,v_k\}\cup \{w_0,\dots,w_{m-1}\}\cup \{w'_0,\dots,w'_{m-1}\},\]
\[E(A_k)=\{uv_1,\dots,uv_k\}\cup \{v_jw_i:\lfloor j/2^i \rfloor \bmod 2 = 0\}\cup \{w_0w'_0,\dots,w_{m-1}w'_{m-1}\},\]
where $m=\lceil \log_2 (1+k)\rceil$. Note that the edge between $v_j$ and $w_i$ exists if the $i$-th bit in the binary representation of $j$ is 0. The proof of the following ``tightness'' proposition is postponed to Appendix \ref{apx:prop:linelower}.
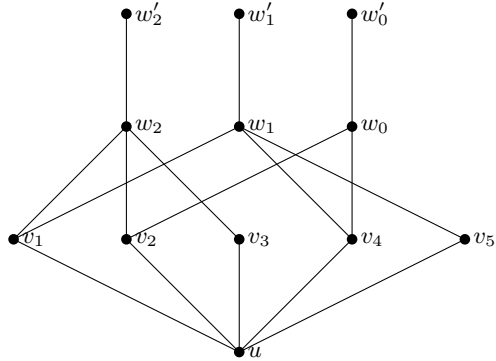
\begin{figure}[ht]
    \centering
    \begin{tikzpicture}
        \coordinate[label=right:{$u$}] (u) at (4.5,0);
        \fill (u) circle[radius=2pt];

        \foreach \i in {1,...,5}
        {
            \coordinate[label=right:{$v_\i$}] (v\i) at (1.5*\i,1.5);
            \fill (v\i) circle[radius=2pt];
            \draw (u) -- (v\i);
        }

        \foreach \i in {0,1,2}
        {
            \coordinate[label=right:{$w_\i$}] (w\i) at (-1.5*\i+6,3);
            \coordinate[label=right:{$w'_\i$}] (ww\i) at (-1.5*\i+6,4.5);
            \fill (w\i) circle[radius=2pt];
            \fill (ww\i) circle[radius=2pt];
            \draw (w\i) -- (ww\i);
        }
        \draw (w0) -- (v2);
        \draw (w0) -- (v4);
        \draw (w1) -- (v1);
        \draw (w1) -- (v4);
        \draw (w1) -- (v5);
        \draw (w2) -- (v1);
        \draw (w2) -- (v2);
        \draw (w2) -- (v3);
    \end{tikzpicture}
    \caption{The graph $A_5$}\label{ak}
\end{figure}

\begin{proposition}\label{linelower}
    For any $k\ge 2$, we have $\Psi(L(A_k))=m=\lceil \log_2 (1+\Delta(A_k))\rceil$. 
\end{proposition}

 \subsection{Upper bound}   
 Before proceeding with the proof for the upper bound $|V(G)|-1$ of $\Psi(L(G))$, we introduce some notations and establish a claim. For any edge $e\in E(G)$, let $E_e(T)=\{e'\in E(T):e\cap e'\neq \emptyset\}$. For any vertex $v\in V(G)$, let $E_v(T)=\{vu:vu\in E(T)\}$. 
    \begin{claim}
    For any edge $xy\in E(G)$, we have $|E_{xy}(T)|\ge 2$.
    \end{claim}
     \begin{proof}
    If $x$ is not a leaf of $T$, then $|E_x(T)|\ge 2$. Since $E_x(T)\subseteq E_{xy}(T)$, the claim holds. Similarly, the claim holds if $y$ is not a leaf of $T$. If both $x$ and $y$ are leaves of $T$, then the edge $xy$ cannot be in $T$ because $|V(G)|\ge 3$. In this case, $|E_x(T)|=|E_y(T)|=1$ and the edges are distinct, so $|E_x(T)\cup E_y(T)|=2$. Since $E_x(T)\cup E_y(T)\subseteq E_{xy}(T)$, the claim is established. \qed
    \end{proof}

    Let $T$ be a spanning tree of $G$. In particular, if $G=K_4$ or $K_4^-$ ($K_4$ with one edge removed), we add the requirement that $T=K_{1,3}$, meaning $T \neq P_4$ (see Fig. \ref{span}). We will now show that $E(T)$ is a DRS for $L(G)$, which implies $\Psi(L(G)) \le |E(T)|=|V(G)|-1$.

    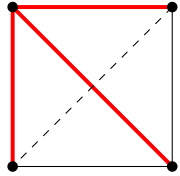
\begin{figure}[ht]
        \centering
        \begin{tikzpicture}
            \coordinate (a0) at (90*0+45:1.5);
            \coordinate (a1) at (90*1+45:1.5);
            \coordinate (a2) at (90*2+45:1.5);
            \coordinate (a3) at (90*3+45:1.5);
            \draw[red,line width=1.5pt] (a0) -- (a1);
            \draw (a2) -- (a3);
            \draw[red,line width=1.5pt] (a1) -- (a2);
            \draw[red,line width=1.5pt] (a1) -- (a3);
            \draw (a0) -- (a3);
            \draw[dashed] (a0) -- (a2);
            \fill (a0) circle[radius=2pt];
            \fill (a2) circle[radius=2pt];
            \fill (a1) circle[radius=2pt];
            \fill (a3) circle[radius=2pt];
        \end{tikzpicture}
        \caption{The case where $G\supseteq K_4^-$. The bold edges represent the chosen spanning tree $T$, and the dashed line represents a possible edge in $G$.}\label{span}
    \end{figure}
 
\begin{proof}[Proof of Theorem \ref{main} (Upper bound)]    Let $a,b$ be two distinct edges in $E(G)$. We must show that $a$ and $b$ are doubly resolved by a pair of edges from $E(T)$. Without loss of generality, we assume that $|E_a(T)|\le |E_b(T)|$.
    If $E_b(T)\setminus E_a(T)\neq \emptyset$, then let $e_1\in E_a(T)\setminus \{b\}$ (which exists by the preceding claim) and $e_2\in E_b(T)\setminus E_a(T)$. We have
    \[d_{L(G)}(a,e_1)-d_{L(G)}(a,e_2)\le 1-2=-1<0=1-1\le d_{L(G)}(b,e_1)-d_{L(G)}(b,e_2),\]
    which implies $\{a,b\}$ is doubly resolved by $\{e_1,e_2\}\subseteq E(T)$.
    
    Otherwise, we consider the case where $E_a(T)=E_b(T)$. Let $a=x_1y_1$ and $b=x_2y_2$. If $\{x_1,y_1\}\cap \{x_2,y_2\}= \emptyset$, then it must be that $|V(G)|\ge 4$. We first assume $|V(G)|\ge 5$. Since $T$ is a spanning tree, there exists an edge $x_3y_3\in E(T)$ such that $x_3\in \{x_1,y_1,x_2,y_2\}$ and $y_3\in V(G)\setminus \{x_1,y_1,x_2,y_2\}$. Without loss of generality, assume $x_3=x_1$. Then $x_3y_3\in E_a(T)\setminus E_b(T)$, a contradiction. Now, assume $|V(G)|=4$. 
    Since $b\notin E_a(T)$ and $a\notin E_b(T)$ and $E_a(T)=E_b(T)$, we have $\{a,b\}\cap E(T)=\emptyset$. 
    Thus $T\subseteq K_{2,2}$, which implies $T=P_4$ and $G\supseteq K_4^-$. This contradicts our choice of $T$.
    
    Now, we assume that $x_1=x_2$. If there exists $x_3\notin \{x_2,y_2\}$ such that $y_1x_3\in E(T)$, then $y_1x_3\in E_a(T)\setminus E_b(T)$, a contradiction. Therefore, we must have
    $E_{y_1}(T)\subseteq \{y_1y_2,x_1y_1\}$. 
    Similarly, we must have 
    $E_{y_2}(T)\subseteq \{y_1y_2,x_2y_2\}$. 
    If $x_1y_1\in E(T)$ and $x_2y_2\in E(T)$, i.e., $\{a,b\}\subseteq E(T)$, then $\{a,b\}$ can be doubly resolved by $\{a,b\}\subseteq E(T)$. If $x_1y_1\notin E(T)$ and $x_2y_2\notin E(T)$, then $y_1y_2\in E(T)$. 
    It implies that $|E_{y_1y_2}(T)|=1$, a contradiction.
    Finally, without loss of generality, we assume that $x_1y_1\in E(T)$ and $y_1y_2\in E(T)$. Let $c=y_1y_2$, then
    \[d_{L(G)}(a,a)-d_{L(G)}(a,c)=-1<0=d_{L(G)}(b,a)-d_{L(G)}(b,c),\]
    which means $\{a,b\}$ is doubly resolved by $\{a,c\}\subseteq E(T)$.\qed
    \end{proof}

\paragraph{Tightness.} The upper bound $|V(G)|-1$ in Theorem \ref{main} is also tight. To see this, consider the graph $T_k$ constructed by gluing $k$ copies of $K_3$  at a single vertex (see Fig. \ref{tk} for an illustration). The proof of the following ``tightness'' proposition is given in Appendix \ref{apx:prop:lineupper}.

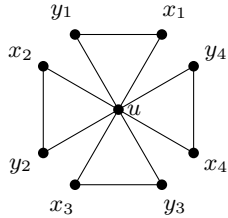
\begin{figure}[ht]
    \centering
    \begin{tikzpicture}
        \coordinate[label=right:{$u$}] (d) at (0,0);
        \fill (d) circle[radius=2pt];
        \foreach \i in {1,...,4}
        {
            \coordinate (b\i) at (90*\i-30:0.577*2);
            \node at (90*\i-30:1.5) {$x_\i$};
            \fill (b\i) circle[radius=2pt];
            \coordinate (c\i) at (90*\i+30:0.577*2);
            \node at (90*\i+30:1.5) {$y_\i$};
            \fill (c\i) circle[radius=2pt];
            \draw (d) -- (b\i) -- (c\i) --  cycle;
        }
    \end{tikzpicture}
    \caption{The graph $T_4$}\label{tk}
\end{figure}

\begin{proposition}\label{lineupper}
    For each $k\ge 1$, $\Psi(L(T_k))=2k=|V(T_k)|-1$.
\end{proposition}

\section{Line graphs of trees}\label{sec:tree}
In this section, we use $T$ to denote a tree. As we have seen, determining $\Psi(L(T))$ requires a different approach than for $\mu(L(T))$. The existing concept of an exterior major vertex is not sufficient to capture the properties of DRSs in this context. We therefore introduce a refined notion of strong exterior major vertex.

A major vertex of $T$ is a \emph{strong exterior major vertex}  if it has a leaf neighbor (see Fig. \ref{tree2}). We denote the set of all the strong exterior major vertices in $T$ by $\mathrm{EX}'(T)$ and set $ex'(T) = |\mathrm{EX}'(T)|$. The goal of this section is to prove  $\Psi(L(T))=\sigma(T)-ex'(T)$ for any tree $T$ (Theorem \ref{treethm}).

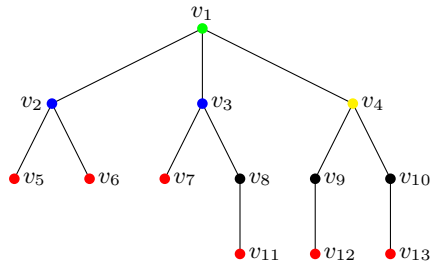
\begin{figure}[ht]
    \centering
    \begin{tikzpicture}
        \coordinate[label=above:$v_1$] (s1) at (0,0);
        \coordinate[label=left:$v_2$] (s2) at (-2,-1);
        \coordinate[label=right:$v_3$] (s3) at (0,-1);
        \coordinate[label=right:$v_4$] (s4) at (2,-1);
        \coordinate[label=right:$v_5$] (s5) at (-2.5,-2);
        \coordinate[label=right:$v_6$] (s6) at (-1.5,-2);
        \coordinate[label=right:$v_7$] (s7) at (-0.5,-2);
        \coordinate[label=right:$v_8$] (s8) at (0.5,-2);
        \coordinate[label=right:$v_9$] (s9) at (1.5,-2);
        \coordinate[label=right:$v_{10}$] (s10) at (2.5,-2);
        \coordinate[label=right:$v_{11}$] (s11) at (0.5,-3);
        \coordinate[label=right:$v_{12}$] (s12) at (1.5,-3);
        \coordinate[label=right:$v_{13}$] (s13) at (2.5,-3);

        \foreach \i[evaluate={\j=int(2*\i+1)}] in {2,3,4}
        {
            \draw (s1) -- (s\i);
            \draw (s\j) -- (s\i);
        }
        \foreach \i[evaluate={\j=int(2*\i+2)}] in {2,3,4}
        {
            \draw (s\j) -- (s\i);
        }
        \foreach \i[evaluate={\j=int(\i+3)}] in {8,9,10}
        {
            \draw (s\j) -- (s\i);
        }
        \foreach \i in {5,6,7,11,12,13}
        {
            \fill[red] (s\i) circle[radius=2pt];
        }
        \foreach \i in {2,3}
        {
            \fill[blue] (s\i) circle[radius=2pt];
        }
        \fill[yellow] (s4) circle[radius=2pt];
        \foreach \i in {8,9,10}
        {
            \fill (s\i) circle[radius=2pt];
        }
        \fill[green] (s1) circle[radius=2pt];
    \end{tikzpicture}
    \caption{An example of special vertices in a tree. The set of major vertices is $\{v_1, v_2, v_3, v_4\}$, the set of external major vertices is $\{v_2, v_3, v_4\}$, and the set of strong external major vertices is $\{v_2, v_3\}$. }\label{tree2}
\end{figure}

 Our proof utilizes a decomposition property of the minimum DRS problem on non-2-connected graphs, established by  \cite{lu2022algorithmic}, which allows us to decompose the problem on the entire graph $L(T)$ into subproblems on its blocks.


A vertex $v$ in a connected graph $G$ is a \emph{cut vertex} if its removal increases the number of connected components of $G$. A connected graph is not 2-connected if and only if it contains at least one cut vertex. A \emph{block} of $G$ is a maximal 2-connected subgraph of $G$. We call a DRS set $S$ a $D$-DRS for every $D\subseteq S$. Given $D\subseteq V(G)$, we define $\Psi_D(G)=\min\{|S|-|D|:S$ is a $D$-DRS of $G\}$. By definition, it is clear that $\Psi(G)-|D| \le \Psi_D(G)\le \Psi(G)$. Note that $\Psi_\emptyset(G)=\Psi(G)$. Lu et~al.~\cite{lu2022algorithmic} established the following decomposition lemma.

\begin{lemma}
\label{lu}
    Let $G=(V,E)$ be a connected graph that is not $2$-connected. Let $G$ be decomposed into its blocks $G_i=(V_i,E_i)$, $i=1,2,\ldots,p$, where $V=\bigcup_{i=1}^{p} V_i$, $E=\bigcup_{i=1}^{p} E_i$, and $E_i\cap E_j=\emptyset$ for $i\neq j$. Let $R$ be the set of cut vertices of $G$ and let $R_i=R\cap V_i$ for $i=1,\dots,p$. Then, a set $S$ is a DRS of $G$ if and only if for every $i\in\{1,\dots,p\}$, the set $S_i=(S\cap V_i)\cup R_i$ is a DRS of $G_i$.
    
    In particular, if each $S_i$ is a minimum $R_i$-doubly resolving set of $G_i$, then $S=\bigcup_{i=1}^{p} (S_i\setminus R_i)$ is a minimum DRS of $G$. Consequently,
    $\Psi(G)=\sum_{i=1}^{p} \Psi_{R_i}(G_i)$.
\end{lemma}

The application of Lemma \ref{lu} to $L(T)$ requires characterizing the cut vertices and blocks of $L(T)$. We accomplish the tasks in the following two lemmas. 

\begin{lemma}\label{treea}  
Let $T$ be a tree with $|V(T)| \geq 3$. An edge $uv\in E(T)$ is a cut vertex of $L(T)$ if and only if neither $u$ nor $v$ is a leaf of $T$.
\end{lemma}  

\begin{proof}  
If one of $u$ and $v$ is a leaf, then without loss of generality, assume $v$ is a leaf. Removing $v$ from $T$ results in the graph $T' = T - \{v\}$, which is also a tree and therefore connected. The vertex set of the graph $L(T) - \{uv\}$ is precisely the edge set of $T'$, i.e., $V(L(T) - \{uv\}) = E(T')$. Since $T'$ is a connected graph with at least one edge (as $|V(T)| \ge 3$), its line graph, $L(T')$, is connected. Thus, $L(T) - \{uv\}$ is connected, which means $uv$ is not a cut vertex of $L(T)$.

Conversely, assume neither $u$ nor $v$ is a leaf. Removing the edge $uv$ from the tree $T$ disconnects it into two non-trivial components (subtrees), let's call them $T_u$ (containing $u$) and $T_v$ (containing $v$). Since neither $u$ nor $v$ is a leaf, they both have a degree of at least 2 in $T$. This ensures that both components $T_u$ and $T_v$ contain at least one edge. The vertex set of $L(T) - \{uv\}$ is the disjoint union of the edge sets $E(T_u)$ and $E(T_v)$. For an edge to exist between two vertices in a line graph, their corresponding edges in the original graph must share a vertex. However, no edge in $T_u$ shares a vertex with any edge in $T_v$. Consequently, in the graph $L(T) - \{uv\}$, there are no edges between any vertex from $E(T_u)$ and any vertex from $E(T_v)$. This means $L(T) - \{uv\}$ is disconnected, making $uv$ a cut vertex.\qed
\end{proof}  


\begin{lemma}\label{treebc}  
Let $T$ be a tree with $|V(T)| \geq 3$. Then the following properties hold:  
\begin{enumerate}[(a)]  
    \item If $u$ is a non-leaf vertex of $T$, then the set of incident edges $E_u= \{e\in E(T): u\in e\}$, which forms a clique in $L(T)$, induces a block $B$ of $L(T)$. \label{treeb}
    \item If $B$ is a block of $L(T)$, then there exists a non-leaf vertex $u$ of $T$ such that $V(B) = E_u$. \label{treec}  
\end{enumerate}  
\end{lemma}  

\begin{proof}  
We first prove \ref{treeb}. Any two edges in $E_u$ are incident at vertex $u$, which means they are adjacent in $L(T)$. Thus, $E_u$ forms a clique in $L(T)$. 

Let $T'$ be a connected subgraph of $T$ such that 
$(\{u\} \cup \{v \mid uv \in E(T)\}) \subsetneqq  V(T')$.
By the connectivity of $T'$, there must exist vertices $v$ and $w$ such that $uv, vw \in E(T')$. In this case, neither $u$ nor $v$ is a leaf in $T'$, and by Lemma \ref{treea}, $uv$ is a cut vertex of $L(T')$. Thus, the subgraph of $L(T)$ induced by $E(T')$, denoted as $B'$, contains a cut vertex. Therefore, $B$ is a maximal subgraph of $L(T)$ that contains no cut vertices, meaning $B$ is a block.  

Now, we prove \ref{treec}. Let $T'$ be a subgraph of $T$ such that $E(T') = V(B)$. Since $B$ is a block, by Lemma \ref{treea}, every edge in $T'$ must be incident to a leaf of $T'$. A tree with this property must be a star, meaning there exists a vertex $u$ such that $E(T') \subseteq E_u$. Moreover, since $B$ is a block, $|E(T')| = |V(B)|\ge 2$, which implies that $u$ is not a leaf. Finally, by \ref{treeb} and the maximality of the block, $E(T') = E_u$. Thus, $V(B) = E_u$.  \qed
\end{proof}  

By Lemma \ref{treebc}, each block of $L(T)$ is a complete graph $K_n$ for some $n$. Recalling the decomposition property (Lemma \ref{lu}), $\Psi(L(T)$ is the sum of a number of $\Psi_{D}(K_n)$. The value of $\Psi_{D}(K_n)$ has been completely determined by Lu et~al.~\cite{lu2022algorithmic} as follows:
\begin{lemma}
\label{block}
Let $K_n$ be a complete graph of order $n\ge 2$. For any $D\subseteq V(K_n)$, 
\[\Psi_{D}(K_n)=\begin{cases}
    n-|D| & \text{if } n=2 \text{ or } D=V(K_n),\\
    n-1-|D| & \text{if } n>2 \text{ and } D\neq V(K_n).\\
\end{cases}\]
\end{lemma}

Having characterized the cut vertices and the block structure of $L(T)$, we are now equipped to prove our main theorem for this section.

\begin{theorem}\label{treethm}  
If $T$ is a tree with $|V(T)| \geq 3$, then $\Psi(L(T)) = \sigma(T) - ex'(T)$.  
\end{theorem}  

\begin{proof}  
Let $u_1, \dots, u_p$ be all the non-leaf vertices of $T$. 
By Lemma \ref{treebc}, the blocks of $L(T)$ are the subgraphs $G_i$ induced by the edge sets $E_{u_i} = \{u_iv : u_iv\in E(T)\}$ for $i=1, \dots, p$.
Let $R$ be the set of cut vertices in $L(T)$. For each block $G_i$, let $V_i = V(G_i)$ and $R_i = R \cap V_i$. By Lemma \ref{treea}, $R_i$ consists of edges $u_iv$ where $v$ is not a leaf, so $R_i = \{u_iv : d(v) \geq 2\}$. Applying the decomposition in Lemma \ref{lu}, we have
$\Psi(L(T)) = \sum_{i=1}^p \Psi_{R_i}(G_i)$.

To evaluate this sum, we analyze $\Psi_{R_i}(G_i)$ by considering the properties of the vertex $u_i$. Let $\sigma_i = |\{u_iv : d(v) = 1\}|$ be the number of leaves adjacent to $u_i$. Note that $\sigma_i = |V_i| - |R_i|$. 
We consider three cases:
\begin{enumerate}[label=\textbf{Case~\arabic*:}~, leftmargin=0pt, labelindent=0pt, labelsep=0em, align=left]
    \item $u_i$ is not a major vertex. This means $d(u_i)=2$, so the block $G_i$ is a $K_2$. According to Lemma~\ref{block}, $\Psi_{R_i}(G_i) = |V_i| - |R_i| = \sigma_i$.

    \item $u_i$ is a major vertex but not a strong exterior major vertex. By definition, $u_i$ is not adjacent to any leaves, so $\sigma_i=0$ and $R_i=V_i$. According to Lemma~\ref{block}, since $D=V(K_n)$, we have $\Psi_{R_i}(G_i) = |V_i|-|R_i|=0$, which is equal to $\sigma_i$.

    \item $u_i$ is a strong exterior major vertex. By definition, $u_i$ is a major vertex ($|V_i| = d(u_i) \geq 3$) and is adjacent to at least one leaf, which implies $R_i \neq V_i$. By Lemma~\ref{block}, we have $\Psi_{R_i}(G_i) = |V_i| - 1 - |R_i| = \sigma_i - 1$.
\end{enumerate}

Therefore, $\Psi(L(T)) = \sum_{i=1}^p \Psi_{R_i}(G_i) = \sigma(T) - ex'(T)$.  \qed
\end{proof}

The proof of Theorem~\ref{treethm} is constructive and yields a linear-time algorithm to find a minimum DRS set of $L(T)$. The algorithm begins by computing the degree of all vertices. An initial candidate set $S$ is then constructed by collecting all edges incident to leaves (degree-1 vertices). During this process, if a parent has a degree of at least 3, it is added to a set of strong exterior major vertices, $\mathrm{EX}'(T)$. Finally, for each vertex $v \in \mathrm{EX}'(T)$, one edge incident to $v$ is removed from $S$, ensuring the final set has cardinality $\sigma(T) - ex'(T)$. As each step can be completed in linear time, the total time complexity to construct a minimum DRS for $L(T)$ is $O(|V(T)|)$.

\section{Conclusion}

In this paper, we investigated the minimum DRS problem for line graphs.  First, we prove that it remains NP-hard even for the restricted class of line graphs of bipartite graphs. Second, for a connected graph $G$ with at least three vertices, we prove tight inequalities $\lceil \log_2 (1+\Delta(G))\rceil \le \Psi(L(G)) \le |V(G)| - 1$. Finally, for any tree $T$, we determine $\Psi(L(T)) = \sigma(T) - ex'(T)$ by introducing the new concept of strong exterior major vertices and giving a linear time algorithm for finding a minimum DRS of $L(T)$. 

While we have solved the problem for line graphs of trees, determining the exact value of $\Psi(L(G))$ for other important graph classes remains an interesting problem.  Furthermore, characterizing the families of graphs for which our upper and lower bounds are met could provide deeper insights into the extremal properties.

\begin{credits}
\subsubsection{\ackname} This research was supported in part by the National Natural Science Foundation of China (NSFC) under grant number 12331014 and the Science and Technology Commission of Shanghai Municipality under grant number 22DZ2229014.

\subsubsection{\discintname}
The author has no competing interests to declare that are relevant to the content of this article.
\end{credits}

\bibliographystyle{splncs04}

\clearpage
\appendix

\section{NP-hardness proof} 
\subsection{Graph construction}\label{apx:graph}
For completeness, we provide a redrawn version of Fig.~7 in \cite{epstein2015weighted}, illustrating the graph $G=(I\cup J,E)$ constructed from the given $3$-dimensional matching instance $(A,B,C,S)$. Note that the graph is actually much larger than what is depicted in Fig. \ref{etunpc}, as each vertex in $A\cup B\cup C\cup S$ should have $N$ copies, whereas the figure displays only one of these $N$ copies.
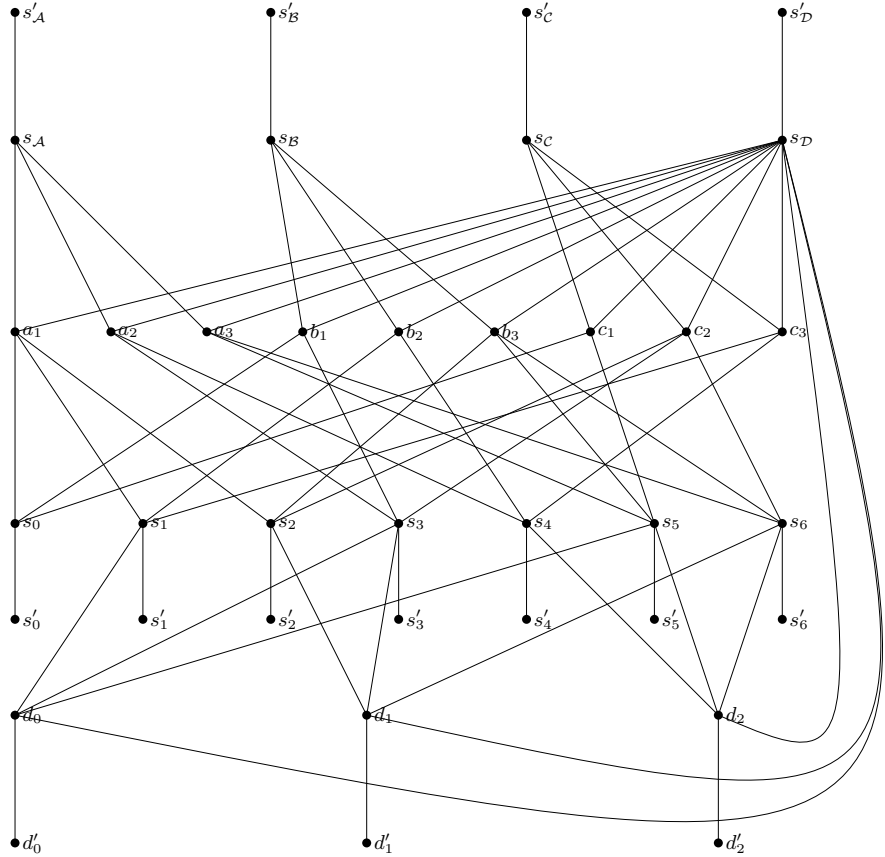
\begin{figure}[H]
\centering
\scalebox{0.85}{
\begin{tikzpicture}

\coordinate[label=right:{$s_{\mathcal{A}}$}] (sa) at (0,0);
\coordinate[label=right:{$s_{\mathcal{B}}$}] (sb) at (4,0);
\coordinate[label=right:{$s_{\mathcal{C}}$}] (sc) at (8,0);
\coordinate[label=right:{$s_{\mathcal{D}}$}] (sd) at (12,0);
\fill (sa) circle[radius=2pt];
\fill (sb) circle[radius=2pt];
\fill (sc) circle[radius=2pt];
\fill (sd) circle[radius=2pt];

\coordinate[label=right:{$s'_{\mathcal{A}}$}] (sa2) at (0,2);
\coordinate[label=right:{$s'_{\mathcal{B}}$}] (sb2) at (4,2);
\coordinate[label=right:{$s'_{\mathcal{C}}$}] (sc2) at (8,2);
\coordinate[label=right:{$s'_{\mathcal{D}}$}] (sd2) at (12,2);
\fill (sa2) circle[radius=2pt];
\fill (sb2) circle[radius=2pt];
\fill (sc2) circle[radius=2pt];
\fill (sd2) circle[radius=2pt];
\draw (sa) -- (sa2);
\draw (sb) -- (sb2);
\draw (sc) -- (sc2);
\draw (sd) -- (sd2);

\foreach \i in {1,2,3}
{
\coordinate[label=right:{$a_\i$}] (a\i) at (1.5*\i-1.5,-3);
\fill (a\i) circle[radius=2pt];

\coordinate[label=right:{$b_\i$}] (b\i) at (1.5*\i+3,-3);
\fill (b\i) circle[radius=2pt];

\coordinate[label=right:{$c_\i$}] (c\i) at (1.5*\i+7.5,-3);
\fill (c\i) circle[radius=2pt];

\draw (sa) -- (a\i);
\draw (sb) -- (b\i);
\draw (sc) -- (c\i);
\draw (sd) -- (a\i);
\draw (sd) -- (b\i);
\draw (sd) -- (c\i);
}
\foreach \i in {0,...,6}
{
\coordinate[label=right:{$s_\i$}] (s\i) at (2*\i,-6);
\fill (s\i) circle[radius=2pt];
}

\foreach \i in {0,...,6}
{
\coordinate[label=right:{$s'_\i$}] (s2\i) at (2*\i,-7.5);
\fill (s2\i) circle[radius=2pt];
\draw (s\i) -- (s2\i);
}

\foreach \i in {0,...,2}
{
\coordinate[label=right:{$d_\i$}] (d\i) at (5.5*\i,-9);
\fill (d\i) circle[radius=2pt];
}
\foreach \i in {0,...,2}
{
\coordinate[label=right:{$d'_\i$}] (d2\i) at (5.5*\i,-11);
\fill (d2\i) circle[radius=2pt];
\draw (d\i) -- (d2\i);
}

\draw (sd) .. controls (40/3,-10) .. (d2);
\draw (sd) .. controls (44/3,-11) .. (d1);
\draw (sd) .. controls (15,-12) .. (d0);

\foreach \i in {0,1,2}
{
\draw (a1) -- (s\i);
}
\foreach \i in {3,4}
{
\draw (a2) -- (s\i);
}
\foreach \i in {5,6}
{
\draw (a3) -- (s\i);
}
\foreach \i in {0,3}
{
\draw (b1) -- (s\i);
}
\foreach \i in {1,4}
{
\draw (b2) -- (s\i);
}
\foreach \i in {2,5,6}
{
\draw (b3) -- (s\i);
}
\foreach \i in {0,5}
{
\draw (c1) -- (s\i);
}
\foreach \i in {2,3,6}
{
\draw (c2) -- (s\i);
}
\foreach \i in {1,4}
{
\draw (c3) -- (s\i);
}

\foreach \i in {1,3,5}
{
\draw (d0) -- (s\i);
}
\foreach \i in {2,3,6}
{
\draw (d1) -- (s\i);
}
\foreach \i in {4,5,6}
{
\draw (d2) -- (s\i);
}

\end{tikzpicture}
}
\caption{An illustration for constructing  bipartite graph $G=(I\cup J,E)$, where $\mathcal{S}=\{(a_1, b_1, c_1)$, $(a_1, b_2, c_3)$, $(a_1, b_3, c_2)$, $(a_2, b_1, c_2)$, $(a_2, b_2, c_3)$, $(a_3, b_3, c_1)$, $(a_3, b_3, c_2)\}$.} \label{etunpc}
\end{figure}

\subsection{Proof of Lemma \ref{edrsnpc}}\label{apx:lma:edrsnpc}
The proof of Lemma~\ref{edrsnpc} is based on the following lemma by Epstein et~al.~\cite{epstein2015weighted} regarding resolving sets.
\begin{lemma}\label{ersnpc}
\begin{enumerate}[(a)]
\item If $L(G)$ has a resolving set $R$ such that $|R|\le K$, then there is a 3-dimensional matching $S'\subseteq S$. \label{enprsa}
\item If there is a 3-dimensional matching $S'\subseteq S$, then the corresponding $\mathcal{S}'\subseteq \mathcal{S}$ gives rise to the set
    \[R=\{s_is_i':s_i\in \mathcal{S}'\}\cup \{s_{\mathcal{A}}s'_{\mathcal{A}}, s_{\mathcal{B}}s'_{\mathcal{B}}, s_{\mathcal{C}}s'_{\mathcal{C}}, s_{\mathcal{D}}s'_{\mathcal{D}}\}\cup \{d_0d'_0,d_1d'_1,\dots,d_{\lambda-1}d'_{\lambda-1}\},\]
    which is a resolving set of $L(G)$ with $|R|=K$. \label{enprsb}
\end{enumerate}
\end{lemma}

The following lemma from \cite{lu2022bridge} establishes a connection between the resolving set and the doubly resolving set. Given a vertex $x\in V(G)$, a vertex subset $S$ of $G$ is a \emph{doubly distance resolving set} of $G$ on $x$ if every pair of vertices $\{u,v\}$ with $d_G(u,x)\neq d_G(v,x)$ is doubly resolved by some pair of vertices in $S\cup \{x\}$.

\begin{lemma}\label{lpsile}
Let $S$ be a resolving set of $G$. Let $x\in S$ and $T$ be a doubly distance resolving set of $G$ on $x$. Then $S\cup T$ is a DRS of $G$.
\end{lemma}

Now we are ready to prove Lemma~\ref{edrsnpc}, which yields the NP-hardness of the minimum DRS problem on the line graphs of bipartite graphs.
\begin{proof}[Proof of Lemma \ref{edrsnpc}]
    Part \ref{enpdrsa} follows directly from Lemma~\ref{ersnpc}\ref{enprsa}, since any DRS is also a resolving set. To prove part \ref{enpdrsb}, by Lemmas~\ref{lpsile} and~\ref{ersnpc}\ref{enprsb}, we only need to show that

    \[R'=\{s_{\mathcal{A}}s'_{\mathcal{A}}, s_{\mathcal{B}}s'_{\mathcal{B}}, s_{\mathcal{C}}s'_{\mathcal{C}}\}\cup \{d_0d'_0,d_1d'_1,\dots,d_{\lambda-1}d'_{\lambda-1}\}\]
    is a doubly distance resolving set on $s_{\mathcal{D}}s'_{\mathcal{D}}$. Let us partition the edge set of $G$ into four sets based on their distance from $s_{\mathcal{D}}s'_{\mathcal{D}}$ in $L(G)$:
    \begin{align*}
        E_0&=\{s_{\mathcal{D}}s'_{\mathcal{D}}\}\\
        E_1&=\{s_{\mathcal{D}}a_i\}\cup \{s_{\mathcal{D}}b_i\}\cup \{s_{\mathcal{D}}c_i\}\cup \{s_{\mathcal{D}}d_k\}\\
        E_2&=\{s_{\mathcal{A}}a_i\} \cup \{s_{\mathcal{B}}b_i\} \cup\{s_{\mathcal{C}}c_i\}\cup \{a_is_j\} \cup \{b_is_j\} \cup \{c_is_j\}\cup \{d_kd'_k\}\cup \{d_ks_j\}\\
        E_3&=\{s_{\mathcal{A}}s'_{\mathcal{A}},s_{\mathcal{B}}s'_{\mathcal{B}},s_{\mathcal{C}}s'_{\mathcal{C}}\}\cup \{s_js'_j\}
    \end{align*}
    To avoid notational clutter, we have omitted the ranges for the indices $i, j,$ and $k$. They are understood to encompass all possible values according to the graph construction.
    It is straightforward to verify that for any $e_k \in E_k$, $d(e_k, s_{\mathcal{D}}s'_{\mathcal{D}}) = k$, and $\bigcup_{k=0}^3 E_k = E(G)$. 
    
    To complete the proof, we must show that any pair of edges $\{e_i, e_j\}$ with $e_i\in E_i$, $e_j\in E_j$, and $i \neq j$, is doubly resolved by $R' \cup \{s_{\mathcal{D}}s'_{\mathcal{D}}\}$. We proceed by considering the following four cases:

    \begin{enumerate}[label=\textbf{Case~\arabic*:}~, leftmargin=0pt, labelindent=0pt, labelsep=0em, align=left]
        \item  
        $e_0=s_{\mathcal{D}}s'_{\mathcal{D}}\in E_0$, $e_1\in E_1$. If $e_1=s_{\mathcal{D}}a_i$, then
        \[d(e_0,e_0)-d(e_0,s_{\mathcal{A}}s'_{\mathcal{A}})=0-3\neq 1-2=d(e_1,e_0)-d(e_1,s_{\mathcal{A}}s'_{\mathcal{A}}),\]
        which shows that $\{e_0,e_1\}$ is doubly resolved by $\{s_{\mathcal{D}}s'_{\mathcal{D}},s_{\mathcal{A}}s'_{\mathcal{A}}\}$. Similarly, if $e_1=s_{\mathcal{D}}b_i$, then $\{e_0,e_1\}$ is doubly resolved by $\{s_{\mathcal{D}}s'_{\mathcal{D}},s_{\mathcal{B}}s'_{\mathcal{B}}\}$. If $e_1=s_{\mathcal{D}}c_i$, then $\{e_0,e_1\}$ is doubly resolved by $\{s_{\mathcal{D}}s'_{\mathcal{D}},s_{\mathcal{C}}s'_{\mathcal{C}}\}$. Finally, if $e_1=s_{\mathcal{D}}d_k$, then
        \[d(e_0,e_0)-d(e_0,d_kd'_k)=0-2\neq 1-1=d(e_1,e_0)-d(e_1,d_kd'_k),\]
        which shows that $\{e_0,e_1\}$  is doubly resolved by $\{s_{\mathcal{D}}s'_{\mathcal{D}},d_kd'_k\}$.

        \item 
        $e_0=s_{\mathcal{D}}s'_{\mathcal{D}}\in E_0$ and $e_2\in E_2$. If $e_2=s_{\mathcal{A}}a_i$, then
        \[d(e_0,e_0)-d(e_0,s_{\mathcal{A}}s'_{\mathcal{A}})=0-3\neq 2-1=d(e_2,e_0)-d(e_2,s_{\mathcal{A}}s'_{\mathcal{A}}),\]
        which shows that $\{e_0,e_2\}$ is doubly resolved by $\{s_{\mathcal{D}}s'_{\mathcal{D}},s_{\mathcal{A}}s'_{\mathcal{A}}\}$. Similarly, if $e_2=s_{\mathcal{B}}b_i$, then $\{e_0,e_2\}$ is doubly resolved by $\{s_{\mathcal{D}}s'_{\mathcal{D}},s_{\mathcal{B}}s'_{\mathcal{B}}\}$. If $e_2=s_{\mathcal{C}}c_i$, then$\{e_0,e_2\}$ is doubly resolved by $\{s_{\mathcal{D}}s'_{\mathcal{D}},s_{\mathcal{C}}s'_{\mathcal{C}}\}$. Now, if $e_2=a_is_j$, then 
        \[d(e_0,e_0)-d(e_0,s_{\mathcal{A}}s'_{\mathcal{A}})=0-3\neq 2-2=d(e_2,e_0)-d(e_2,s_{\mathcal{A}}s'_{\mathcal{A}}),\]
        which shows that $\{e_0,e_2\}$ is doubly resolved by $\{s_{\mathcal{D}}s'_{\mathcal{D}},s_{\mathcal{A}}s'_{\mathcal{A}}\}$. Similarly, if $e_2=b_is_j$, then $\{e_0,e_2\}$ is doubly resolved by $\{s_{\mathcal{D}}s'_{\mathcal{D}},s_{\mathcal{B}}s'_{\mathcal{B}}\}$. If $e_2=c_is_j$, then $\{e_0,e_2\}$ is doubly resolved by $\{s_{\mathcal{D}}s'_{\mathcal{D}},s_{\mathcal{C}}s'_{\mathcal{C}}\}$. Next, if $e_2=d_kd'_k$, then it is obvious that $\{e_0,e_2\}$ is doubly resolved by $\{e_0,e_2\}$. Finally, if $e_2=d_ks_j$, then 
        \[d(e_0,e_0)-d(e_0,d_kd'_k)=0-2\neq 2-1=d(e_2,e_0)-d(e_2,d_kd'_k),\]
        which shows that $\{e_0,e_2\}$ is doubly resolved by $\{s_{\mathcal{D}}s'_{\mathcal{D}},d_kd'_k\}$. 
        
        \item 
        $e_1\in E_1$, $e_2\in E_2$. Suppose $e_1=s_{\mathcal{D}}a_i$. Then $\{e_1,e_2\}$ is doubly resolved by $\{s_{\mathcal{D}}s'_{\mathcal{D}},s_{\mathcal{A}}s'_{\mathcal{A}}\}$ if and only if
        \[d(e_1,s_{\mathcal{D}}s'_{\mathcal{D}})-d(e_1,s_{\mathcal{A}}s'_{\mathcal{A}})=1-2\neq 2-d(e_2,s_{\mathcal{A}}s'_{\mathcal{A}})=d(e_2,s_{\mathcal{D}}s'_{\mathcal{D}})-d(e_2,s_{\mathcal{A}}s'_{\mathcal{A}}),\]
        which shows that $\{e_1,e_2\}$ is doubly resolved by $\{s_{\mathcal{D}}s'_{\mathcal{D}},s_{\mathcal{A}}s'_{\mathcal{A}}\}$ if and only if $d(e_2,s_{\mathcal{A}}s'_{\mathcal{A}})\neq 3$. It is easy to verify that
        \begin{align*}
            d(s_{\mathcal{A}}a_i,s_{\mathcal{A}}s'_{\mathcal{A}})&=1,\\
            d(a_is_j,s_{\mathcal{A}}s'_{\mathcal{A}})&=2,\\
            d(b_is_j,s_{\mathcal{A}}s'_{\mathcal{A}})&=d(c_is_j,s_{\mathcal{A}}s'_{\mathcal{A}})=d(d_ks_j,s_{\mathcal{A}}s'_{\mathcal{A}})=3,\\
            d(s_{\mathcal{B}}b_i,s_{\mathcal{A}}s'_{\mathcal{A}})&=d(s_{\mathcal{C}}c_i,s_{\mathcal{A}}s'_{\mathcal{A}})=d(d_kd'_k,s_{\mathcal{A}}s'_{\mathcal{A}})=4.
        \end{align*}
        Therefore, $\{e_1,e_2\}$ is doubly resolved by $\{s_{\mathcal{D}}s'_{\mathcal{D}},s_{\mathcal{A}}s'_{\mathcal{A}}\}$ if and only if $e_2\in \{s_{\mathcal{A}}a_i\} \cup \{s_{\mathcal{B}}b_i\} \cup\{s_{\mathcal{C}}c_i\}\cup \{a_is_j\} \cup \{d_kd'_k\}$. Similarly, if $e_2=b_is_j$, $\{e_1,e_2\}$ is doubly resolved by $\{s_{\mathcal{D}}s'_{\mathcal{D}},s_{\mathcal{B}}s'_{\mathcal{B}}\}$. if $e_2=c_is_j$, $\{e_1,e_2\}$ is doubly resolved by $\{s_{\mathcal{D}}s'_{\mathcal{D}},s_{\mathcal{C}}s'_{\mathcal{C}}\}$. Finally, if $e_2=d_ks_j$, then 
        \[d(e_1,s_{\mathcal{D}}s'_{\mathcal{D}})-d(e_1,d_kd'_k)=1-2\neq 2-1=d(e_2,s_{\mathcal{D}}s'_{\mathcal{D}})-d(e_2,d_kd'_k),\]
        which shows that $\{e_1,e_2\}$ is doubly resolved by $\{s_{\mathcal{D}}s'_{\mathcal{D}},d_kd'_k\}$. For the cases where $e_1 = s_{\mathcal{D}}b_i$ or $e_1 = s_{\mathcal{D}}c_i$, the proof follows similarly.
        Now assume $e_1=s_{\mathcal{D}}d_k$. Then $\{e_1,e_2\}$ is doubly resolved by $\{s_{\mathcal{D}}s'_{\mathcal{D}},s_{\mathcal{A}}s'_{\mathcal{A}}\}$ if and only if
        \[d(e_1,s_{\mathcal{D}}s'_{\mathcal{D}})-d(e_1,s_{\mathcal{A}}s'_{\mathcal{A}})=1-3\neq 2-d(e_2,s_{\mathcal{A}}s'_{\mathcal{A}})=d(e_2,s_{\mathcal{D}}s'_{\mathcal{D}})-d(e_2,s_{\mathcal{A}}s'_{\mathcal{A}}),\]
        which shows that $\{e_1,e_2\}$ is doubly resolved by $\{s_{\mathcal{D}}s'_{\mathcal{D}},s_{\mathcal{A}}s'_{\mathcal{A}}\}$ if and only if $d(e_2,s_{\mathcal{A}}s'_{\mathcal{A}})\neq 4$. Therefore, $\{e_1,e_2\}$ is doubly resolved by $\{s_{\mathcal{D}}s'_{\mathcal{D}},s_{\mathcal{A}}s'_{\mathcal{A}}\}$ if and only if $e_2\in \{s_{\mathcal{A}}a_i\} \cup \{a_is_j\} \cup \{b_is_j\} \cup \{c_is_j\}\cup \{d_ks_j\}$. Similarly, if $e_2=s_{\mathcal{B}}b_i$, $\{e_1,e_2\}$ is doubly resolved by $\{s_{\mathcal{D}}s'_{\mathcal{D}},s_{\mathcal{B}}s'_{\mathcal{B}}\}$. If $e_2=s_{\mathcal{C}}c_i$, $\{e_1,e_2\}$ is doubly resolved by $\{s_{\mathcal{D}}s'_{\mathcal{D}},s_{\mathcal{C}}s'_{\mathcal{C}}\}$. Finally, if $e_2=d_kd'_k$, then 
        \[d(e_1,s_{\mathcal{D}}s'_{\mathcal{D}})-d(e_1,d_kd'_k)=1-1\neq 2-0=d(e_2,s_{\mathcal{D}}s'_{\mathcal{D}})-d(e_2,d_kd'_k),\]
        which shows that $\{e_1,e_2\}$ is doubly resolved by $\{s_{\mathcal{D}}s'_{\mathcal{D}},d_kd'_k\}$. 
        
        \item 
        $e\in E_0\cup E_1\cup E_2$, $e_3\in E_3$. If $e_3\in \{s_{\mathcal{A}}s'_{\mathcal{A}},s_{\mathcal{B}}s'_{\mathcal{B}},s_{\mathcal{C}}s'_{\mathcal{C}}\}$, then 
        \[d(e,s_{\mathcal{D}}s'_{\mathcal{D}})-d(e,e_3)\le 2-1< 3-0=d(e_3,s_{\mathcal{D}}s'_{\mathcal{D}})-d(e_3,e_3),\]
        which shows that $\{e,e_3\}$ is doubly resolved by $\{s_{\mathcal{D}}s'_{\mathcal{D}},e_3\}$. If $e_3=s_js'_j$, then
        \[d(e_3,s_{\mathcal{D}}s'_{\mathcal{D}})-d(e_3,s_{\mathcal{A}}s'_{\mathcal{A}})=3-3=0.\]
        It is not hard to check that $d(e,s_{\mathcal{D}}s'_{\mathcal{D}})-d(e,s_{\mathcal{A}}s'_{\mathcal{A}})=0$ if and only if $e=a_is_j$. 
        In other words, if $e$ is not of the form $a_is_j$, then $\{e, e_3\}$ is doubly resolved by $\{s_{\mathcal{D}}s'_{\mathcal{D}}, s_{\mathcal{A}}s'_{\mathcal{A}}\}$. 
        Finally, by symmetry, if $e=a_is_j$, then $\{e,e_3\}$ is doubly resolved by $\{s_{\mathcal{D}}s'_{\mathcal{D}},s_{\mathcal{B}}s'_{\mathcal{B}}\}$. \qed
    \end{enumerate}
\end{proof}

\section{Proof of Proposition \ref{linelower}}\label{apx:prop:linelower}
\begin{proof}
    We begin by establishing equalities and inequalities that will be essential for the proof. Firstly, we prove that \( k \ge m \). Let \( f(k) = k - \log_2(1+k) \), then
    \[f'(k)=1-\frac{1}{\ln 2\cdot (1+k)}\ge 1-\frac{1}{\ln 8}>0.\]
    Therefore, $k-m=\lfloor f(k) \rfloor\ge \lfloor f(2) \rfloor=0$. 
    Secondly, by the definition of \( m \), we know that $m-1<\log_2 (1+k)\le m$, hence $2^{m-1}<1+k\le 2^m$. From this, we get $2^{m-1}\le k<2^m$, so $m-1=\lfloor\log_2 k \rfloor$. 
    
    To prove that $\Psi(L(A_k))\ge m$, it suffices to show that $\Delta(A_k)=k$, as the result then follows directly from Theorem \ref{main}.
    First, $d(u)=k$ and $d(w'_i)=1$. Secondly, for each $1\le j\le k$, there must be a binary digit that is 1, that is, there exists $i$, such that $\lfloor j/2^i \rfloor \bmod 2 = 1$. 
    In other words, for any $v_j$, there exists $w_i$, such that $v_jw_i\notin E$. Therefore, $d(v_j)\le (m-1)+1=m\le k$. Finally, since for any $0\le i\le m-1$, we have $1\le 2^i\le 2^{m-1}\le k$, and there exists $1\le j\le k$, such that $j=2^i$. That is, for any $w_i$, there must exist $v_j$, such that $v_jw_i\notin E$. Hence, $d(w_i)\le (k-1)+1=k$. 

    Next, we show that $\Psi(L(A_k))\le m$ by constructing a DRS of size $m$. Let $S=\{w_0w'_0,\dots,w_{m-1}w'_{m-1}\}$. We will prove that $S$ is a DRS for $L(A_k)$. The proof relies on a case analysis of pairs of edges, using the following distance formulas within $L(A_k)$:  
    \[
    \begin{aligned}
        d(w_iw'_i, w_{i'}w'_{i'}) &= 
        \begin{cases}
            0 & \text{if } i = i', \\
            3 & \text{if } i \neq i' \text{ and there exists } j \text{ such that } w_iv_j, w_{i'}v_j \in E, \\
            5 & \text{otherwise}.
        \end{cases} \\
        d(w_iw'_i, w_{i'}v_j) &= 
        \begin{cases}
            1 & \text{if } i = i', \\
            2 & \text{if } i \neq i' \text{ and } w_iv_j \in E, \\
            3 & \text{if } w_iv_j \notin E \text{ and there exists } j' \text{ such that } w_iv_{j'}, w_{i'}v_{j'} \in E, \\
            4 & \text{otherwise}.
        \end{cases} \\
        d(w_iw'_i, uv_j) &= 
        \begin{cases}
            2 & \text{if } w_iv_j \in E, \\
            3 & \text{if } w_iv_j \notin E.
        \end{cases}
    \end{aligned}
    \]  
    
    We first clarify the ``otherwise'' cases in our distance formulas, where distances of 4 or 5 can occur. These cases require that for a pair of distinct indices $i, i'$, there is no vertex $v_j$ such that $w_iv_j, w_{i'}v_j \in E$.
    We claim this scenario is impossible when $m \ge 3$. If $m \ge 3$, for any distinct $i, i'$, we can choose a third index $i'' \notin \{i, i'\}$. 
    Let $j = 2^{i''}$. Since $i, i' \neq i''$, the $i$-th and $i'$-th bits of $j$ are both 0. Thus, $w_iv_j, w_{i'}v_j \in E$, which is a contradiction.
    Therefore, these cases only occur when $m = 2$ (i.e., when $k \in \{2, 3\}$). For these scenarios, we can verify directly (see Figs.~\ref{a2} and \ref{a3}).  

    \begin{figure}[ht]
        \centering
        \begin{minipage}[b]{0.45\textwidth}
        \centering
        \begin{tikzpicture}
        \coordinate[label=right:{$u$}] (u) at (2.25,0);
        \fill (u) circle[radius=2pt];

        \foreach \i in {1,...,2}
        {
            \coordinate[label=right:{$v_\i$}] (v\i) at (1.5*\i,1.5);
            \fill (v\i) circle[radius=2pt];
        }

        \foreach \i in {0,1}
        {
            \coordinate[label=right:{$w_\i$}] (w\i) at (-1.5*\i+3,3);
            \coordinate[label=right:{$w'_\i$}] (ww\i) at (-1.5*\i+3,4.5);
            \fill (w\i) circle[radius=2pt];
            \fill (ww\i) circle[radius=2pt];
        }
        \draw (w0) -- node[right] {$5$} (ww0);
        \draw (w0) -- node[right] {$3$} (v2);
        \draw (v2) -- node[below right] {$1$} (u);
        \draw (u) -- node[below left] {$-1$} (v1);
        \draw (w1) -- node[left] {$-3$} (v1);
        \draw (w1) -- node[left] {$-5$} (ww1);
        \end{tikzpicture}
        \caption{The graph $A_2$, where the number on edge $e$ represents $d(e,w_1w'_1) - d(e,w_0w'_0)$. Since the numbers on each edge are all different, it indicates that $S = \{w_0w'_0, w_1w'_1\}$ is a DRS of $L(A_2)$.}\label{a2}
        
        \end{minipage}
        \quad
        \begin{minipage}[b]{0.45\textwidth}
        \centering
        \begin{tikzpicture}
        \coordinate[label=right:{$u$}] (u) at (3,0);
        \fill (u) circle[radius=2pt];

        \foreach \i in {1,...,3}
        {
            \coordinate[label=right:{$v_\i$}] (v\i) at (1.5*\i,1.5);
            \fill (v\i) circle[radius=2pt];
        }

        \foreach \i in {0,1}
        {
            \coordinate[label=right:{$w_\i$}] (w\i) at (-1.5*\i+3,3);
            \coordinate[label=right:{$w'_\i$}] (ww\i) at (-1.5*\i+3,4.5);
            \fill (w\i) circle[radius=2pt];
            \fill (ww\i) circle[radius=2pt];
        }
        \draw (w0) -- node[right] {$5$} (ww0);
        \draw (w0) -- node[right] {$3$} (v2);
        \draw (v2) -- node[right] {$1$} (u);
        \draw (u) -- node[below left] {$-1$} (v1);
        \draw (u) -- node[below right] {$0$} (v3);
        \draw (w1) -- node[left] {$-3$} (v1);
        \draw (w1) -- node[left] {$-5$} (ww1);
        \end{tikzpicture}
        \caption[$A_3$]{The graph $A_3$, where the number on edge $e$ represents $d(e,w_1w'_1) - d(e,w_0w'_0)$. Since the numbers on each edge are all distinct, it indicates that $S = \{w_0w'_0, w_1w'_1\}$ is a DRS of $L(A_3)$.}\label{a3}
        \end{minipage}
    \end{figure}

    For the remainder of the proof, we assume $m \ge 3$. Let $e_1, e_2$ be two distinct edges in $E$. We analyze the following six cases:
    \begin{enumerate}[label=\textbf{Case~\arabic*:}~, leftmargin=0pt, labelindent=0pt, labelsep=0em, align=left]
        \item  $e_1 = w_iw'_i$, $e_2 = w_{i'}w'_{i'}$.  
        This case is straightforward, as $\{e_1, e_2\}$ can be doubly resolved by $\{e_1, e_2\}\subseteq S$.  

        \item  $e_1 = w_iw'_i$, $e_2 = w_{i'}v_j$.  
        Select any $i'' \neq i, i'$ and let $e_3 = w_{i''}w'_{i''}$. Then,  
        \[
        d(e_1, e_1) - d(e_1, e_3) = -3 < 1 - 3 \leq d(e_2, e_1) - d(e_2, e_3),
        \]  
        which shows that $\{e_1, e_2\}$ can be doubly resolved by $\{e_1, e_3\}\subseteq S$.  

        \item  $e_1 = w_iw'_i$, $e_2 = uv_j$.  
        Select any $i' \neq i$ and let $e_3 = w_{i'}w'_{i'}$. Then,  
        \[
        d(e_1, e_1) - d(e_1, e_3) = -3 < 2 - 3 \leq d(e_2, e_1) - d(e_2, e_3),
        \]  
        which shows that $\{e_1, e_2\}$ can be doubly resolved by $\{e_1, e_3\}\subseteq S$.  

        \item  $e_1 = w_iv_j$, $e_2 = w_{i'}v_{j'}$.  
        If $i \neq i'$, then let $e_3 = w_iw'_i$ and $e_4 = w_{i'}w'_{i'}$. We have 
        \[
        d(e_1, e_3) - d(e_1, e_4) \leq 1 - 2 < 2 - 1 \leq d(e_2, e_3) - d(e_2, e_4),
        \]  
        which shows that $\{e_1, e_2\}$ can be doubly resolved by $\{e_3, e_4\}\subseteq S$. If $i = i'$, then $j \neq j'$. There exists a binary digit differing between $j$ and $j'$, say $i''$, which means exactly one of $w_{i''}v_j \in E$ and $w_{i''}v_{j'} \in E$ holds. Without loss of generality, we assume that $w_{i''}v_j \in E$ and $w_{i''}v_{j'} \notin E$. Let $e_3 = w_iw'_i$ and $e_4 = w_{i''}w'_{i''}$. Then,  
        \[
        d(e_1, e_3) - d(e_1, e_4) = 1 - 2 > 1 - 3 = d(e_2, e_3) - d(e_2, e_4),
        \]  
        which shows that $\{e_1, e_2\}$ can be doubly resolved by $\{e_3, e_4\}\subseteq S$.  

            \item  $e_1 = w_iv_j$, $e_2 = uv_{j'}$.  
        Select $i'$ such that $w_{i'}v_j \notin E$, and let $e_3 = w_iw'_i$, $e_4 = w_{i'}w'_{i'}$. Then,  
        \[
        d(e_1, e_3) - d(e_1, e_4) = 1 - 3 < 2 - 3 \leq d(e_2, e_3) - d(e_2, e_4),
        \]  
        which shows that $\{e_1, e_2\}$ can be doubly resolved by $\{e_3, e_4\}\subseteq S$.  

        \item  $e_1 = uv_j$, $e_2 = uv_{j'}$.  
        Since $j \neq j'$, their binary representations must differ at some bit position, say $i$. This implies that exactly one of $w_iv_j \in E$ and $w_iv_{j'} \in E$ holds. Without loss of generality, we assume that $w_iv_j \in E$ and $w_iv_{j'} \notin E$. Additionally, there exists $i'$ such that $w_{i'}v_j \notin E$. Let $e_3 = w_iw'_i$ and $e_4 = w_{i'}w'_{i'}$. Then,  
        \[
        d(e_1, e_3) - d(e_1, e_4) = 2 - 3 < 3 - 3 \leq d(e_2, e_3) - d(e_2, e_4),
        \]  
        which shows that $\{e_1, e_2\}$ can be doubly resolved by $\{e_3, e_4\}\subseteq S$. \qed
    \end{enumerate}
\end{proof}

\section{Proof of Proposition \ref{lineupper}}\label{apx:prop:lineupper}
\begin{proof}
    Let $u$ be the center vertex of $T_k$ and $x_i,y_i$ be the other two vertices in each $K_3$. Let $S$ be the DRS of $L(T_k)$. We shall prove that $|S\cap \{ux_i,uy_i,x_iy_i\}|\ge 2$ in later. This implies that $|S|\ge 2k$. The result then follows from Theorem \ref{main}.

        We proceed by contradiction, assuming that $|S\cap \{ux_i,uy_i,x_iy_i\}|\le 1$. First, we assume that $S\cap \{ux_i,uy_i\}=\emptyset$. Note that
    \[d_{L(T_k)}(ux_i,x_iy_i)-d_{L(T_k)}(uy_i,x_iy_i)=1-1=0.\]
    For each $j\neq i$, we have
    \[d_{L(T_k)}(ux_i,ux_j)-d_{L(T_k)}(uy_i,ux_j)=1-1=0,\]
    and
    \[d_{L(T_k)}(ux_i,x_jy_j)-d_{L(T_k)}(uy_i,x_jy_j)=2-2=0.\]
    This means that $\{ux_i,uy_i\}$ cannot be doubly resolved by $S$, a contradiction.

    Now, we have $S\cap \{ux_i,uy_i\}\neq \emptyset$. Since we assumed $|S\cap \{ux_i,uy_i,x_iy_i\}|\le 1$, it must be that $|S\cap \{ux_i,uy_i\}|=1$ and $x_iy_i\notin S$. By symmetry, we can assume without loss of generality that $S\cap \{ux_i,uy_i\}=\{ux_i\}$. Note that
    \[d_{L(T_k)}(x_iy_i,ux_i)-d_{L(T_k)}(ux_i,ux_i)=1-0=1.\]
    For each $j\neq i$, we have
    \[d_{L(T_k)}(x_iy_i,ux_j)-d_{L(T_k)}(ux_i,ux_j)=2-1=1\]
    and
    \[d_{L(T_k)}(x_iy_i,x_jy_j)-d_{L(T_k)}(ux_i,x_jy_j)=3-2=1.\]
    This means that $\{x_iy_i,ux_i\}$ cannot be doubly resolved by $S$, a contradiction. \qed
\end{proof}

\end{document}